\documentclass[11pt, reqno]{amsart}

\usepackage{geometry}
\usepackage{color}
\usepackage{amssymb}
\usepackage{amsmath}
\usepackage{arydshln}
\usepackage{hyperref}
\usepackage{bbm}
\usepackage{mathrsfs}
\usepackage{enumerate}

\usepackage{graphicx}
\usepackage{subfig}

\numberwithin{equation}{section}

\newcommand{\R}{\mathbb{R}}

\renewcommand{\P}{\mathbb{P}}
\newcommand{\E}{\mathbb{E}}

\newcommand{\cal}[1]{\mathcal{#1}}

\newtheorem{theorem}{Theorem}[section]

\newtheorem{lemma}[theorem]{Lemma}
\newtheorem{remark}[theorem]{Remark}
\newtheorem{definition}[theorem]{Definition}

\begin{document}

\title{Stability properties of mild solutions of SPDEs related to  Pseudo  Differential Equations} 


\author{Vidyadhar Mandrekar}
\address[Vidyadhar Mandrekar]{Department of Statistics and Probability\\
Michigan State University\\
East Lansing, MI, USA}

\author{Barbara R\"udiger}
\address[Barbara R\"udiger]{School of Mathematics and Natural Sciences\\ University of Wuppertal, Germany}
\email[Barbara R\"udiger]{ruediger@uni-wuppertal.de}

 \begin{abstract}  This is a review article which presents part of the contribution of Sergio Albeverio to the study of existence and uniqueness of solutions of SPDEs driven by jump processes  and their stability properties. The results on stability properties obtained in Albeverio et al. \cite{AGMRS2017} are presented in a slightly  simplified and  different way.
 \end{abstract}

 \maketitle

\\

\noindent {\bf AMS classification 2020: } 60H15,  60G51, 37L40, 60J76

\section{Introduction}\label{s:Intro}
 The theory of SPDE's driven by Brownian motion was studied for a long time and solutions taking values in a Hilbert space are described in \cite{DZbook}, \cite{GMbook} based on previous work. Sergio Albeverio was among the  first mathematicians to initiate the study of SPDE driven by jump processes \cite{AWZ1998} with solutions in Hilbert spaces in contrast  to Kallianpur and Xiong \cite{KXbook}  who studied generalized solutions. In order to study these equations in general, Sergio et al. provided the L\'evy -It\^{o} decomposition in Banach spaces \cite{AR2005}. There was a previous approach by E. Dettweiler \cite{Dettweiler1983}, where the stochastic integrals are defined differently from those of  It\^{o}. In  \cite{AR2005} it was  however proven that the definitions are equivalent.
 Starting from   \cite{AR2005}(M- type 2 and type 2) Banach valued stochastic integrals with respect to L\'evy processes and compensated Poisson random measures associated to additive processes were defined in \cite{Ru2004}, \cite{MR2009}, \cite {MRbook}, including the case of separable Hilbert space valued stochastic integrals. (For the theory of stochastic integration on Banach spaces see  also \cite{PZbook}, \cite{vNVWbook}  and references there) 
 Following this, the article  \cite{AMR2009} establishes the basic generalization  of classical work for mild solutions  of SPDE's driven by L\'evy processes and associated Poisson random noise.
 An It\^{o}-formula was proved in this case \cite{RZ2006} which was later generalized in \cite{MRT2013} and further in \cite{AGMRS2017}. It has interesting  applications to stability of solutions of such SPDEs which originated in \cite{MRT2013}  and have been continued in \cite{MRbook} and  by Albeverio et al.  in \cite{AGMRS2017}.\\
 Our project in this paper is to present first a review of the work mentioned. The results related to the stability properties obtained in \cite{AGMRS2017} are presented in Section 4 in  a simplified and slightly different  way, by involving a dissipativity condition (condition i) in Theorem \ref{Thdissipativity}). These motivated further investigations of  stability properties of SPDEs with multiple invariant measures in \cite{FFRSch2020}, which introduces  a ''generalized dissipativity condition'',  that are not reported in this paper due to stipulated page limitations  for this article.\\
 
 \section{Stochastic Integrals and It\^{o} -formula}
  Consider a filtered probability space $(\Omega, \mathcal{F}, \{\mathcal{F}_t\}_{t\geq 0}, \mathbb{P})$ satisfying the usual conditions.
 Let $H$ be a separable Hilbert space  with norm $\|\cdot\|_H$ and scalar product $<\cdot,\cdot>_H$, which for simplicity we will often denote with $\|\cdot\|$  and $<\cdot,\cdot>$. Let  $\{L_t\}_{t \geq 0}$ be an $H$ - valued L\'evy process on $(\Omega, \mathcal{F}, \{\mathcal{F}_t\}_{t\geq 0}, \mathbb{P})$. Let $\mathcal{B}(H)$ denote the Borel -$\sigma$ - Algebra on $H$. For $B$$\in \mathcal{B}(H)$ with $0 \notin \overline{B}$, we define 
 $$
 N((0,t]\times B)= \sum_{0< s \leq t}  \mathbbm{1}_{B}(\Delta L_s) \quad t \geq 0
 $$
and 
$$  N((0,t]\times \{0\})= 0 $$
 We define 
 \begin{eqnarray*}
 \beta:\,\,\mathcal{B}(H) & \, \to\,& \mathbb{R}_+ :=\{t \in \mathbb{R}: t \geq 0\}\\
 B & \, \to\,&\beta(B):=\mathbb{E}[N((0,1]\times B]
 \end{eqnarray*}
 Observe that the random measure $N(dt, dx)$ induced on  $ \mathcal{B}(\mathbb{R}_+)  \otimes \mathcal{B}(H)$  is a Poisson random measure with compensator $\nu(dt,dx)$$:=dt \otimes \beta(dx)$.  We recall that on the trace $\sigma $ -algebra   $ \mathcal{B}(\mathbb{R}_+)  \otimes \mathcal{B}(H\setminus \{0\})$ the  Poisson random measure and its compensator are $\sigma$ -finite measures , but might not be finite, since the jumps of the underlying  L\'evy process  $\{L_t\}_{t \geq 0}$ in a time interval $[0,T]$ are numerable, but might not be finite. These are however finite on each set $[0,T]\times A$  with $A\in \mathcal{B}(H)$,  $0\notin  \overline{A}$ . We shall be dealing with non -Gaussian L\'evy processes, i.e. 
 $$L_t - \int_0^t \int_{\|x\|_H \leq 1} x N(dt,dx)=0 \quad \forall t \geq 0$$ 
 (See e.g. Proposition 3.3.8 in \cite{MRbook} or  Section 4.5 in \cite{PZbook}.)\\
 
 \noindent We denote with  $q(dt,dx):=N(dt,dx)-\nu(dt,dx)$  the compensated Poisson random measure associated to $N(dt,dx)$. \\
 
 \noindent Let us remark that $M_\cdot:=$$\{M_t\}_{t \geq 0} $    with
 $M_t:=q((0,t]\cap A\times B) $
 is   for each   $A\in\mathcal{B}(\mathbb{R}_+)$ and $ B\in \mathcal{B}(H) $ with $\beta(B)< \infty$ a $\{\mathcal{F}_t\}_{t \geq 0}$ - martingale  (See e.g. Lemma 2.4.7 in \cite{MRbook}.)\\
 
 \noindent Let $F$ be a separable Hilbert  space. 
 Let ${\rm Ad}(F)$ denote the space of all functions $f:\,
  \mathbb{R}_+ \times H\times \Omega\, \to\, F$ which are  adapted on the enlarged space 
 $$
(\tilde{\Omega},\tilde{\mathcal{F}},\tilde{\mathcal{F}}_t,\tilde{\mathbb{P}}) = (\Omega \times H,\mathcal{F} \times \mathcal{B}(H),(\mathcal{F}_t \times \mathcal{B}(H))_{t \geq 0},\mathbb{P} \otimes \beta).
$$
 We can define the It\^{o} - Integral of $f$ w.r.t. the compensated Poisson random meansure $q(ds,dx)$  basically starting with simple funtions which are square integrable w. r. t. $\mathbb{P} \otimes \beta$ and then  by density arguments  for all $f\in L_{{\rm ad}}^2(F):= L^2(\tilde{\Omega} \times   \mathbb{R}_+ , \tilde{\mathcal{F}} \otimes \mathcal{B}(  \mathbb{R}_+ ),\mathbb{P} \otimes \beta\otimes \lambda; F) \cap {\rm Ad}(F) $, where $\lambda$ denotes the Lebesgues -measure. (See  e.g Section 3.5 in \cite{MRbook}.)\\
 
\noindent The  It\^{o} - Integral of $f$ w.r.t. the compensated Poisson random meansure $q(ds,dx)$ 
\begin{equation} \label{Ito integral} Z_t:= \int_0^t \int_A f(s,x) q(ds,dx) \quad t\geq 0 \end{equation}
is then  a square integrable martingale for all $A\in \mathcal{B}(H)$ such that $0 \notin \overline{A}$.\\

\noindent  Through stopping times the  It\^{o} - Integral  (\ref{Ito integral}) can be extended also to all $f\in$ $\mathcal{K}_{\infty,\beta}^2(F)$ which denotes  the linear space of all progressively measurable functions $f : \mathbb{R}_+ \times  H  \times \Omega  \rightarrow F$ such that
 
 $$\mathbb{P}\left(\int_0^t \int_H \|f(s,x)\|^2  q(ds,dx)<\infty \right) =1 \quad \forall t\geq 0.$$
 
\noindent  The It\^{o} -Integral (\ref{Ito integral}) is then  a local martingale (see e.g. \cite{MRbook} Section 3.5, where this theory has been discussed for $F$ being a separable  Banach space).\\

\noindent Let us present the It\^{o} -Formula for the It\^{o} - Process
 $(Y_t)_{t \geq 0}$, with
\begin{equation}\label{Ito process}
Y_t:=Z_t+\int_0^t F_s ds+ \int_0^t \int_\Lambda k(s,x)
N(ds,dx),
\end{equation}

\noindent where $(Z_t)_{t \geq 0}$ is defined through equation  (\ref{Ito integral}),  $F_\cdot:=$$\{F_t\}_{t \geq 0} $  is an  $F$ -valued $\{\mathcal{F}_t\}_{t \geq 0} $ adapted process, which satisfies 
 $$\mathbb{P}\left(\int_0^t \|F_s\| \, ds\,<\infty \right) =1 \quad \forall t\geq 0, $$ 
$\Lambda\in\mathcal{B}(H)$ is a set with $\beta(\Lambda) < \infty$, 
$k : \Omega \times \mathbb{R}_+ \times H \rightarrow F$ is a progressively measurable process. Moreover
$k$ is c\`adl\`ag or c\`agl\`ad $\beta(dx) \otimes  \mathbb{P}$--almost surely and 
 \begin {eqnarray*} \int_0^t \int_{\Lambda} \| k(s,x) \|  \, \nu(ds,dx) \, <\infty \quad \mathbb{P} \,-a.s..\end {eqnarray*}\\

\noindent Let  ${\mathcal H} \in C_b^{1,2}(\mathbb{R}_+ \times H; F)$, the space of functions ${\mathcal H}\,:  \mathbb{R}_+ \times H \, \to \, F$ which are differentiable in $t\in \mathbb{R}_+$ and twice Fr\'ech\'et  differentiable in $x \in H$, with bounded derivatives. Then similar to  \cite{IWbook} (for the finite dimensional case) it can be proven that the following It\^{o} -Formula holds . (See e.g. \cite{Abook} or  \cite{RZ2006} for the Banach valued  case.)
\begin{theorem} \label{Ito} Let $A \in \mathcal{B}(H)$.
Assume 
 \begin {eqnarray*} \int_0^t \int_{A} \| f(s,x) \| \nu(ds,dx) \, <\infty \quad \mathbb{P} \,-a.s..\end {eqnarray*} 
 or 
 \begin {eqnarray*} \int_0^t \int_{A} \| f(s,x) \|^2 \nu(ds,dx) \, <\infty \quad \mathbb{P} \,-a.s..\end {eqnarray*}
\begin{enumerate}
\item We have $\mathbb{P}$--almost surely

\begin{eqnarray*}
&{\mathcal H} (t,Y_t) = {\mathcal H} (0,Y_0) + \int_0^t \partial_s {\mathcal H} (s,Y_{s}) ds + \int_0^t \partial_y {\mathcal H} (s,Y_{s}) F_s ds 
\\&\quad + \int_0^t \int_A \big( {\mathcal H} (s,Y_{s-} + f(s,x)) - {\mathcal H} (s,Y_{s-}) \big) q(ds,dx)
\\ &\quad + \int_0^t\int_A \big( {\mathcal H} (s,Y_{s} + f(s,x))- {\mathcal H} (s,Y_{s})- \partial_y {\mathcal H} (s,Y_{s})f(s,x) \big) \nu(ds,dx)
\\ &\quad + \int_0^t \int_{\Lambda} \big( {\mathcal H} (s,Y_{s-}
+ k(s,x)) - {\mathcal H} (s,Y_{s-}) \big) N(ds,dx), \quad t \geq 0.
\label{Ito-formula}\end{eqnarray*}
where

\item for all $t \in \mathbb{R}_+$ we have $\mathbb{P}$--almost surely
\begin{eqnarray*}
&\int_0^t \| \partial_s {\mathcal H} (s,Y_{s}) \| ds < \infty, 
\\  &\int_0^t \int_A \|{\mathcal  H} (s,Y_{s} + f(s,x)) - {\mathcal H} (s,Y_{s}) \|^2 \nu(ds,dx) < \infty,
\\  &\int_0^t \int_A \|{\mathcal H } (s,Y_{s} + f(s,x))-
{\mathcal H} (s,Y_{s})- \partial_y {\mathcal H} (s,Y_{s})f(s,x) \| \nu(ds,dx) < \infty,
\\  &\int_0^t \int_{\Lambda} \| {\mathcal H}(s,Y_{s-} + k(s,x))- {\mathcal H} (s,Y_{s_-}) \| N(ds,dx) < \infty.
\end{eqnarray*}

\end{enumerate}
\end{theorem}

\noindent However we are often interested in applying the It\^{o} formula to functions ${\mathcal H}$ which are only in $  C^{1,2}(\mathbb{R}_+ \times H; H)$, i.e where  the Fr\'ech\'et derivatives are not necessarily  bounded. Especially for stochastic models  applied to physics we might be interested in taking advantage of conservation of energy of a random process and would like to compute $\|Y_t\|^2$. 
Remark however that $H(y) = \| y \|^2$ is of class $C^2(H;\mathbb{R})$.

\noindent Let us define 
\begin{definition}
 A continuous, non-decreasing function $h : \mathbb{R}_+ \rightarrow \mathbb{R}_+$ is \emph{quasi-sublinear}  if there is a constant $C > 0$ such that
\begin{eqnarray*}
h(x+y) &\leq C \big( h(x) + h(y) \big), \quad x,y \in \mathbb{R}_+,
 \\ h(xy) &\leq C h(x) h(y), \quad x,y \in \mathbb{R}_+.
\end{eqnarray*}
\end{definition}

\noindent  In \cite{MRT2013} the following was proved:
\begin{theorem}\label{thm-Ito}
Let us assume 
\begin{enumerate}
\item[a)] ${\mathcal H} \in C^{1,2}(\mathbb{R}_+ \times H; F)$ is a function such that
\begin{eqnarray*}\label{qe-1}
\| \partial_y {\mathcal H} (s,y) \| \leq h_1(\| y \|), \quad (s,y) \in \mathbb{R}_+ \times H
\\ \label{qe-2} \| \partial_{yy}{\mathcal  H} (s,y) \| \leq h_2(\| y \|), \quad (s,y) \in \mathbb{R}_+ \times H
\end{eqnarray*}
for quasi-sublinear functions $h_1,h_2 : \mathbb{R}_+ \rightarrow \mathbb{R}_+$.

\item[b)] $f : H  \times \mathbb{R}_+ \times \Omega \rightarrow F$ is a progressively measurable process such that for all $t \in \mathbb{R}_+$ we have $\mathbb{P}$--almost surely
\begin{eqnarray*}
&\int_0^t \int_A \| f(s,x) \|^2 \nu(ds,dx) + \int_0^t \int_A h_1( \| f(s,x) \|)^2 \| f(s,x) \|^2  \nu(ds,dx) 
\\& \quad + \int_0^t \int_A h_2( \| f(s,x) \|) \| f(s,x) \|^2 \nu(ds,dx) < \infty.
\end{eqnarray*}

\end{enumerate}
Then the It\^{o} -Formula 1. with  2. holds.
\end{theorem}
\begin{remark} We remark that $H(y) = \| y\|^2$ is of class $C^2(H;\mathbb{R})$  and 
 \begin{equation}
 H_y(y)v = 2 < y, v >\quad and \quad H_{yy}(y)(v)(w) = 2 < v,w >.
\end{equation}
 so that if for all $t \in \mathbb{R}_+$ we have $\mathbb{P}$--almost surely  $\int_0^t \int_A \| f(s,x) \|^2 \nu(ds,dx)< \infty$ and  $\int_0^t \int_A \| f(s,x) \|^4 \nu(ds,dx)< \infty$, then Theorem \ref{thm-Ito} can be applied to 
$\mathcal{H} (s,y):=H(y)=\|y\|^2$ .
\end{remark}
 
 



\section{SPDEs on Hilbert spaces }
In this section we shall be studying Stochastic Partial Differential Equations (SPDEs) driven by L\'evy processes. Let $(H, \|\cdot\|_H)$ be a Hilbert space and $A$ be an infinitesimal generator of a semigroup $\{S_t, t \geq 0\}$ on $H$ to $H$. This means 
\begin{enumerate}
\item[i)] $S_0= I$
\item[ii)] $S_{s+t}=S_s S_t \quad \forall s,t\geq 0$
\end{enumerate}
We also assume that  $\{S_t, t \geq 0\}$ is strongly continuous, i.e. 
\begin{enumerate}
\item[iii)] $\lim_{t \to 0} S_t x=0$ (in norm $\|\cdot\|_H$)
\end{enumerate}
 If  $\{S_t, t \geq 0\}$ is a semigroup satisfying the above properties , we call it a ''strongly continuous semigroup'' ($C_0$ - Semigroup). For such a semigroup we note that there exists $\alpha \geq 0$ and $M\geq 1$  such that the operator norm in the space $L(H)$ of bounded linear operators from $H$ to $H$   satisfies 
  $$\quad \|S_t\|_{L(H)}\leq M e^{\alpha t}\quad t \geq 0.$$
  We call the semigroup $\{S_t, t \geq 0\}$ ''pseudo -contraction'' semigroup if $M=1$,   ''uniformly bounded  -semigroup'' if $\alpha =0$ and ''contraction semigroup'' if $M=1$ and $\alpha =0$.  \\
  If $t\, \to \, S_t$ is differentiable  for all $x \in H$  then the semigroup  $\{S_t, t \geq 0\}$ is differentiable.\\
  
 \noindent  Let $\{S_t\}$$:=\{S_t, t \geq 0\}$ be a $C_0$ -semigroup on $H$. The linear operator $A$ with domain 
$$ 
\mathcal {D}(A):=\{ x \in H, \lim_{t \to {0^+}} \frac{S_t x-x} {t} \quad exists\}$$
defined by \index{D@${\cal D}(A)$}
$$ Ax = \lim_{t \to {0^+}} \frac{S_t x-x} {t} $$
is called the infinitesimal generator (i.g.) of  $\{S_t\}$. \index{semigroup!infinitesimal generator}\\

\noindent The following facts for an i.e.  $A$ of a $C_0$ -semigroup $\{S_t\}$ are well known (see e.g. \cite{Pazy}):\\
\begin{enumerate}
\item [(1)] \quad  For $x \in H$ 
$ \lim_{h \to 0} \frac {1}{h} \int_t^{t+h} S_s x ds =S_t x$.
\item [(2)] \quad For $x \in {\cal D}(A)$, $S_t x\in {\cal D}(A)$ and 
$\frac {d} {dt} S_t x= A S_t x=S_t A x$.
\item [(3)] \quad  For $x \in H$, $\int_0^t S_s x ds \in {\cal D}(A)$ and 
 $ A \int_0^t S_s x ds=S_t x -x$.
\item [(4)] \quad $ {\cal D}(A) $ is dense in $H$ and $A$ is a closed operator.
\item[(5)] \quad Let  $f:\, [0,T] \,\to\, \mathcal{D}(A)$ be a measurable function with $\int_0^T \|f(s)\|_{\mathcal{D}(A)} ds< \infty $, then $\int_0^T f(s) ds \in \mathcal{D}(A)$  and 
$\int_0^T A f(s) ds = A \int_0^T f(s) ds$. 
\end{enumerate}
We associate with $A$ the resolvent set $\rho(A)$ as the set of complex numbers $\lambda$ for which $\lambda I -A$ has bounded inverse 
$$R(\lambda, A):= (\lambda I -A)^{-1} \in L(H)$$
and we call $R(\lambda, A), \, \lambda \in \rho(A)$ the resolvent of $A$.\\

\noindent We note that $R(\lambda,A):\,H \to {\cal D}(A) $ is one- to -one, i.e. 
\begin {eqnarray*}
&(\lambda I -A) R(\lambda, A) x  = x, \quad x \in H \\ and & \, R(\lambda, A) (\lambda I -A) x = x, \quad x \in {\cal D}(A),   \\{giving} &\, A R(\lambda, A) x  =  R(\lambda, A) A x , \quad x \in {\cal D}(A) 
\end {eqnarray*}

\noindent Remark that $R(\lambda_1,A) R(\lambda_2, A)$$=R(\lambda_2,A) R(\lambda_1, A)$ for $\lambda_1$, $\lambda_2$ $\in {\cal D}(A)$.

 \begin{lemma} \label{lemma -resolvent} Let $\{S_t\}$ be $C_0$ -semigroup with infinitesimal generator $A$. Let 
$$
\alpha_0:= \lim_{t \to \infty} t^{-1} ln(\|S_t\|_{{ L}(H)}), $$
 then any real number $\lambda > \alpha_0$ belongs to the resolvent set $\rho(A)$ and 
 $$ 
 R(\lambda,A) x=\int_0^\infty e^{-\lambda t} S_t x dt \quad x \in E
 $$
 In addition for $x \in H$ 
 $$ \lim_{\lambda \to \infty} \|\lambda R(\lambda,A) x - x\|_H=0$$
\end{lemma}

\begin{theorem} \label{theorem -Hille -Yosida} {\bf Hille - Yosida Theorem} \index{Hille - Yosida Theorem} 
Let $A:\, {\cal D}(A) \subset H \to H$ be a linear operator on a Hilbert space $H$. Necessary and sufficient conditions for $A$ to generate a $C_0$ -semigroup is 
\begin {enumerate}
\item [(1)] \quad $A$ is closed and $\overline {{\cal D}(A)}$$=H$
\item [(2)]\quad There exists $\alpha$, $M$ $\in \mathbb{R}$ such that for $\lambda > \alpha$, $\lambda \in $ $\rho(A)$ 
$$ \|R(\lambda,A)^r\|_{{ L}(H)} \leq M (\lambda - \alpha)^{-r}, \quad r=1,2,...$$
\end {enumerate}
In this case $\|S_t\|_{{ L}(H)} \leq M e^{\alpha t}$, $t \geq 0$.
\end{theorem}

\noindent For $\lambda \in \rho (A)$, consider the family of operators 
$$ R_\lambda:=\lambda R(\lambda,A) .$$
Since the range ${\cal R}(R(\lambda,A))$ of $R(\lambda,A)$ is such that  ${\cal R}(R(\lambda,A)) \subset {\cal D}(A)$, we define the "Yosida approximation" of $A$ by \index{Yosida approximation} 
$$
A_\lambda x = A R_\lambda  x , \quad x \in H
$$
Using  $\lambda (\lambda I -A) R(\lambda, A) =\lambda I$ it is easy to prove 
$$ A_\lambda x = \lambda^2 R(\lambda, A) - \lambda I,\quad A_\lambda\in L(H) $$
\noindent Denote by $S^\lambda_t$ the uniformly continuous semigroup 
$$ S^\lambda_t x = e^{t A_\lambda} x, \quad x \in H $$
Using the commutativity of the resolvent, we get $A_{\lambda_1} A_{\lambda_2} = A_{\lambda_2} A_{\lambda_1}$, and clearly 
$$ A_\lambda S^\lambda_t = S^\lambda_t A_\lambda $$

\begin {theorem} \label {theorem -Yosida -approx} {\bf Yosida approximation} Let $A$ be an infinitesimal generator of a $C_0$ -semigroup $\{S_t\}$ on a Hilbert space $H$. Then 
\begin {enumerate}
\item [a)] $\quad \lim_{\lambda \to \infty} R_\lambda x=x, \quad x\in H$
\item [b)]  $\quad  \lim_{\lambda \to \infty}  A_\lambda x = A x ,\quad  {for}\quad x \in {\cal D} (A)$ 
\item [c)] $\quad \lim_{\lambda \to \infty} S^\lambda_t  x =S_t x, \quad x \in H$
\end {enumerate}
The convergence in c) is uniform on compact subsets of $\mathbb{R}_+$ and 
$$\|S^\lambda_t\|_{{ L}(H)} \leq M exp \left(\frac{t \wedge \alpha}{\lambda -\alpha}\right)$$
with constants $M$, $\alpha$ as in Hille -Yosida Theorem
\end {theorem}

\smallskip \noindent  We conclude this section by introducing a concept of solution. Let us look at the deterministic problem 
$$
\frac {d u(t)}{dt} =A u(t), \,\,  u(0)=x, \quad x\in H
$$
Here $H$ is a real separable Hilbert space and $A$ is an unbounded operator generating a $C_0$ -semigroup.

\noindent A classical solution $u:\, [0,T] \to H$ of the above equation will require a solution to be continuously differentiable and $u(t)$$\in {\cal D}(A)$. However, 
$$ u^x(t)=S_t x , \quad t\geq 0$$
is considered as a (mild) solution to the equation (\cite {Pazy}, Capt.4). 

\noindent One can consider the non -homogeneous equation 
$$
\frac {d u(t)}{dt} =A u(t) + f(t, u(t)), \,\, u(0)=x, \quad x \in H
$$
then for $f\in L^1([0,T],H)$, Bochner integrable, one can consider the integral equation 
\begin{equation}\label{mild det} u^x(t)=S_t x+\int_0^t S_{t-s} f(s, u(s)) ds \end{equation}
A solution of (\ref{mild det}) is called a "mild solution", if $u\in C([0,T],H)$.\\

\noindent Motivated by the initial work of Sergio Albeverio with Wu and Zhang \cite{AWZ1998}, we continued with Sergio \cite{AMR2009}  and further in  \cite{MRbook} to analyze   mild solutions of stochastic partial differential equations (SPDEs) with Poisson noise on any filtered probability space  $(\Omega, \mathcal {F}, \{\mathcal{F}_t\}_{t \geq 0},\mathbb{P})$, satisfying the usual conditions 
 with values on a  separable Hilbert space $H$.  (For this topic see also the monograph by Peszat and Zabczyk \cite{PZbook} and references there.) Remark that 
the stochastic integral $\int_0^t S_{t-s} f(s,x) q(ds,dx)$, which appears in such SPDEs, is in general not a martingale. However similar to Doob inequalities the following Lemma  holds. 

 \begin{lemma}\label{lemma-conv-S2} [Lemma 5.1.9 \cite{MRbook}]
 Assume  $\{S_t\}_{t \geq 0}$  is pseudo-contractive. Let $q(ds, dx) $ be a compensated Poisson random measure on $\mathbb{R}_+ \times E$, for some Hilbert space $E$, associated to a Poisson random measure $N$ with compensator $dt \otimes \beta(dx)$ on $(\Omega, \mathcal{F}, \{\mathcal{F}_t\}_{t\geq 0}, \mathbb{P})$.
 For each $T \geq 0$ the following statements are valid:

 \begin{enumerate}
 \item There exists a constant $C > 0$ such that for each $f \in {L}_{ad}^2(H)$ we have
 \begin{eqnarray}
 &\mathbb{E} \bigg[ \sup_{t \in [0,T]} \bigg\| \int_0^t \int_E S_{t-s} f(s,x) q(ds,dx) \bigg\|^2 \bigg]  \leq \nonumber \\& C e^{2 \alpha T} \mathbb{E} \bigg[ \int_0^T \int_E
  \| f(s,x) \|^2 \beta(dx)ds \bigg]. \label{Ichikawa-E}
 \end{eqnarray}
 
 \item For all  $f \in {L}_{ad}^2(H)$  and all $\epsilon > 0$ we have
 \begin{eqnarray}
 &\mathbb{P} \bigg[ \sup_{t \in [0,T]} \bigg\| \int_0^t \int_E  S_{t-s} f(s,x) q(ds,dx) \bigg\| > \epsilon \bigg]\\& \leq \frac{4 e^{2 \alpha T}}{\epsilon^2} \mathbb{E} \bigg[ \int_0^T \int_E \| f(s,x) \|^2 \beta(dx)ds \bigg]. \label{Ichikawa-P}
\nonumber  \end{eqnarray}
 where $\int_0^t S_{t-s} f(s,x) q(ds,dx)$ is well defined, if the right side  is finite. $\int_0^t S_{t-s} f(s,x) q(ds,dx)$ is c\`ad\`ag.
 
 \end{enumerate}
 \end{lemma}

\medskip \noindent Let us assume that we are given
 
 \smallskip \noindent
 \begin{equation} 
 F:\, H \,\to \,H\,,\end {equation}

 \begin{equation} 
 f:\, H  \times H \,\to \,H\,,
 \end{equation}

 \smallskip \noindent Assume
 
 \smallskip \noindent A)\quad $ \,f(u,z)\,$ is jointly
 measurable,
 
 \smallskip \noindent B)\quad $ \,F(z)\,$ is measurable,

  \smallskip \noindent C) $\,$ there exist  constants $L_f$ and $L_F>0$, s.th.
 
 \begin {eqnarray*} 
 &\|F(z)-F(z')\|^2  \leq L_F \|z -z'\|^2 \\& \int_H\|f(u,z) -f(u,z')\|^2 \beta(du)
 \leq L_f \|z -z'\|^2  \\&
  for \quad all \quad z,\,z'\in H\,
 \end {eqnarray*}

\smallskip \noindent D) 
\begin{equation} 
\int_H \|f(u,0)\|^2 \beta(du) < \infty
\end{equation}

\smallskip \noindent  E) $\quad$ 
 $\,{ A}\,$ is the infinitesimal generator of a pseudo
- contraction semigroup $\,\{S_t\}_{t \in [0,T]}\,$.

\medskip \noindent Remark that Assumptions C) and  D) imply that there is a constant $K>0$ such that 
\begin{equation}
\int_H \|f(u,z)\|^2 \beta(du)\leq K  (1 +\|z \|^2)< \infty, 
\end{equation}

\noindent since

 \begin {eqnarray*} 
 &\int_H \|f(u,z)\|^2 \beta(du)\leq 2 \int_H \|f(u,z)-f(u,0)\|^2 \beta(du)+2 \int_H \|f(u,0)\|^2 \beta(du)\\& \leq 
2 max\{ L_f, \int_H \|f(u,0)\|^2 \beta(du) \} (1 +\|z \|^2)< \infty
\end{eqnarray*}

\medskip \noindent In  Albeverio et al. \cite{AMR2009} and  \cite{MRbook},  we analyzed  (in more generality than in Theorem \ref{theorem Markov-SPDE} below) the existence and uniqueness of mild solutions of the stochastic differential equation on intervals $[0,T]$, $T>0$, like e.g. 

\begin{eqnarray} \label{SPDE-Hilbert}
dX_t  &=&  ( A X_t + F(X_t )) dt + \int_{H} f(u,X_t) q(dt,du) \\ X_0 &=&  \xi. 
\end{eqnarray}

\noindent where $q(dt,du):=N(dt,du) - dt \beta(du)$ is a compensated Poisson random measure with compensator $\nu(dt, du):=  dt \beta(du)$.

\noindent In other words, we looked  at the solution of the integral equation
\begin {equation} \label {mild SPDE}
X_t =  S_t X_0 + \int_0^t S_{t-s} F(X_s) ds + \int_0^t \int_{H} S_{t-s} f(u,X_s) q(ds,du) 
\end {equation}
where integrals on the r.h.s. are well defined \cite{MRbook}.

\begin{definition}\label{imlddf1}
A stochastic process ${X_\cdot}$ is called a mild solution of (\ref{SPDE-Hilbert}), if for all $t\leq T$ \\
(i) $X_t$ is $\mathcal{F}_{t}$-adapted on a filtered probability space ($\Omega,\mathcal{F},\left\{\mathcal{F}_{t}\right\}_{t\leq T},\mathbb{P}$), \\
(ii) $\left\{X_t,t\geq0\right\}$ is jointly measurable and $\int_{0}^{T}E\left\|X_t\right\|_{H}^{2}dt<\infty$, \\
(iii) $X_\cdot$ satisfies   (\ref {mild SPDE}) $\mathbb{P}$ -a.s. on $[0,T]$.

\end{definition}

\begin{definition}\label{imlddf2}
A stochastic process ${X_\cdot}$ is called a strong solution of  (\ref{SPDE-Hilbert}), if for all $t\leq T$\\
(i) $X_t$ is $\mathcal{F}_{t}$-adapted on a filtered probability space ($\Omega,\mathcal{F},\left\{\mathcal{F}_{t}\right\}_{t\leq T},\mathbb{P}$), \\
(ii) $X_\cdot $ is c\`{a}dl\`{a}g with probability one,\\
(iii) $X_t\in\mathcal{D}(A)$, $dt\otimes d\mathbb{P}$ a.e., $\int_{0}^{T}\left\|AX_t\right\|_{H}dt<\infty$  $\,\mathbb{P}$ -a.s.,\\
(iv) $X_\cdot$ satisfies   (\ref{SPDE-Hilbert}) $\mathbb{P}$ -a.s. on $[0,T]$.
\end{definition}

\noindent Obviously, a strong solution $X_\cdot $  of (\ref{SPDE-Hilbert}) is a a mild solution  of (\ref{SPDE-Hilbert}). The contrary is not neccesserily true, since e.g.  $X_t\in\mathcal{D}(A)$ might not be true. (See e.g. Section 2.2 in Albeverio et al. \cite{AGMRS2017} where  sufficient conditions for a mild solution $X_\cdot$ of  (\ref{SPDE-Hilbert}) are listed, for  $X_\cdot$ to be also a strong solution.)

 \medskip \noindent
 Let $S_T^2$  be the linear space of all c\`{a}dl\`{a}g, adapted processes $X_\cdot$ such that
 \begin{equation}
 \mathbb{E} \bigg[ \sup_{t \in [0,T]} \| X_t \|_F^2 \bigg] < \infty,
 \end{equation}
 where we identify processes whose paths coincide almost surely.
 Note that, by the completeness of the filtration, adaptedness does not depend on the choice of the representative.
 
 \begin{lemma}\label{lemma-S2-Banach-space} [Lemma 4.2.1 \cite{MRbook}] 
 The linear space $S_T^2$, equipped with the norm
 \begin{equation}
 \| X_\cdot\|_{S_T^2} = \mathbb{E} \bigg[ \sup_{t \in [0,T]} \| X_t \|^2 \bigg]^{1/2},
 \end{equation}
 is a Banach space.
 \end{lemma}

\begin {theorem} \label {theorem Markov-SPDE} [Theorem 5.3.1 \cite{MRbook}]
Suppose assumptions A) -E) are satisfied. Then for $\xi$$\in L^2(\Omega,\mathcal{F}_0,\mathbb{P};H)$  and $T>0$, there exists a unique mild solution $X_\cdot^\xi$ in $ S_T^2$ to (\ref{SPDE-Hilbert}) with initial condition $\xi$, and satisfying  $X_t^\xi$  is $\mathcal{F}_t$ -measurable.
\end {theorem}

\begin{remark}\label{Remark Feller}
 For  each $\xi, \eta, \in L^2(\Omega, \mathcal{F}_0, \mathbb{P}; H)$,
the corresponding unique solutions $X_\cdot ^\xi$ and $Y_\cdot^\eta$  to (\ref{SPDE-Hilbert}) in Theorem \ref{theorem Markov-SPDE} satisfy
\begin{equation}\label{eq: continuous dependence on initial condition}
 \mathbb{E} \left[ \| X_t - Y_t \|_H^2 \right] \leq C(T)\E\left[ \| \xi-\eta \|_H^2 \right], \quad t \in [0,T].
\end{equation}
for some constant $C(T)$ depending on $T>0$ (See Section 5.7 in \cite{MRbook}).\\
If $X_0 \equiv x \in H$, then   the corresponding solution $X_\cdot ^x$ to (\ref{SPDE-Hilbert})  in Theorem \ref{theorem Markov-SPDE} is Markov (See Section 5.4 in \cite{MRbook}) . Such solution constitutes a Markov process whose transition probabilities $p_t(x,dy) = \P[ X_t^x \in dy]$
are measurable with respect to $x$. By slight abuse of notation we denote by $(p_t)_{t \geq 0}$ its transition semigroup, i.e., for each bounded measurable function $f: H \longrightarrow \R$, $p_tf$ is given by
\begin{equation}\label{transition prob}
 p_tf(x) = \E\left[ f(X_t^x) \right] = \int_{H}f(y)p_t(x,dy), \quad t \geq 0, \ \ x \in H.
\end{equation}
Since due to (\ref{eq: continuous dependence on initial condition})  the solution  dependences continuosly  on the initial condition,   
it can be shown that $p_tf \in C_b(H)$ for each $f \in C_b(H)$,
i.e. the transition semigroup is $C_b$-Feller. 
 \end{remark}

\smallskip \noindent Let $R_{n}=nR(n,A)$, with $n \in \mathbb{N}$, $n\in\rho(A)$, the resolvent set of $A$,
 $R(n,A)=(nI-A)^{-1}$ . The (SPDE) 
\begin{eqnarray} \label{imldeq2}
dX_t  &=&  ( A X_t +R_n F(X_t )) dt + \int_{H} R_nf(u,X_t) q(dt,du) \\ X_0 &=&  R_n \xi(\omega). \nonumber
\end{eqnarray}
obtained by  Yosida Approximation  of  (\ref{SPDE-Hilbert})  has a unique strong solution  $X^{n,\xi}_\cdot$ which approximates its mild solution  $X^{ \xi}_\cdot$  of  (\ref{SPDE-Hilbert}) with initial condition $X^\xi_0=\xi$. The precise statement is given in the following Theorem:

\begin{theorem}\label{imldtm3}
 Suppose assumptions A) -E) are satisfied. Then for $\xi$$\in L^2(\Omega,\mathcal{F}_0,\mathbb{P};H)$ and $T>0$, there exists a unique strong  solution  $X^{n,\xi}_\cdot$ $:=\left\{ X^{n,\xi}_t, \,t\geq0\right\}$ in $ S_T^2$ to (\ref{imldeq2}) with initial condition $\xi$, and satisfying  $X^{n,\xi}_t$  is $\mathcal{F}_t$ -measurable $\,\forall t\geq 0$. Moreover, 

\begin{equation}\label{imldeq5}
\lim_{n\rightarrow\infty}E\bigg[\sup_{0\leq t\leq T}\left\|X^{n,\xi}_t-X^\xi_t\right\|_{H}^{2}\bigg] =0,
\end{equation}
where $X^\xi_\cdot:=$ $\left\{ X^\xi_t, \,t\geq0\right\}$ is the mild solution of equation (\ref{SPDE-Hilbert}) with initial condition $\xi$.
\end{theorem}

\noindent For the proof see Theorem 2.9 of  Albeverio et al. \cite{AGMRS2017}.\\

\begin{definition} $X^{n,\xi}_\cdot$ is called ``the Yosida approximation of $X^\xi_\cdot$''.
\end{definition}

\begin{remark}\label {RemarkIto}
 Let $\,{ A}\,$ be  the infinitesimal generator of a pseudo
- contraction semigroup $\,\{S_t\}_{t \in [0,T]}\,$.
  Assume that $X_\cdot$ is a strong solution of (\ref{SPDE-Hilbert}) and   all the hypotheses in Theorem \ref{thm-Ito} are  satisfied. Then the It\^{o} -Formula  holds and can be written in the following way:\\

 \noindent  $\mathbb{P}$--almost surely

\begin{eqnarray*}
&{\mathcal H} (t,X_t) = {\mathcal H} (0,X_0) + \int_0^t \partial_s {\mathcal H} (s,X_{s}) ds+ \label{ItoLyapunov} \\& \int_0^t \mathcal{L}\mathcal{H}(s,X_s) ds + \int_0^t \int_A \big( {\mathcal H} (s,X_{s-} + f(s,u)) - {\mathcal H} (s,X_{s-}) \big) q(ds,du)  \end{eqnarray*}
with

\begin{eqnarray} \label{Lyapunov}  &\mathcal{L}\mathcal{H}(s,x):=<\partial_x {\mathcal H} (s,x) , Ax+F(x)>+ \\& \int_H \big( {\mathcal H} (s,x+ f(s,u))- {\mathcal H} (s,x)- <\partial_x {\mathcal H} (s,x),f(s,u)> \big) \beta(du) \nonumber
 \end{eqnarray}
\end{remark}

\begin{remark}\label {RemarkItoYosidaApp} Assume that hypotheses A)-E) and   all  hypotheses a) and b)  in Theorem \ref{thm-Ito} are  satisfied. Then the It\^{o} -Formula  for the Yosida approximation $X_\cdot^{n,\xi}$ of the mild solution  $X_\cdot^{\xi}$ of (\ref{SPDE-Hilbert}) holds and can be written in the following way:

\begin{eqnarray*}
&{\mathcal H} (t,X^{n,\xi}_t) = {\mathcal H} (0,X^{n,\xi}_0) + \int_0^t \partial_s {\mathcal H} (s,X^{n,\xi}_{s}) ds+ \label{ItoLyapunovYosida} \\& \int_0^t \mathcal{L}_n\mathcal{H}(s,X^{n,\xi}_s)  ds+ \int_0^t \int_A \big( {\mathcal H} (s,X^{n,\xi}_{s-} + R_n f(s,u)) - {\mathcal H} (s,X^{n,\xi}_{s-}) \big) q(ds,du)  \end{eqnarray*}
with

\begin{eqnarray*} &\mathcal{L}_n\mathcal{H}(s,x):=<\partial_x {\mathcal H} (s,x) , Ax+R_nF(x)>+ \\& \int_H \big( {\mathcal H} (s,x +R_n f(s,u))- {\mathcal H} (s,x)- <\partial_x {\mathcal H} (s,x),R_nf(s,u)> \big) \beta(du) \label{LyapunovYosida} 
 \end{eqnarray*}

\noindent This follows directly from Theorem \ref{imldtm3} and Remark \ref {RemarkIto}.
\end{remark}

\smallskip \noindent In the next Section we will use the following result, which was obtained in \cite{AGMRS2017} as a consequence of an It\^o -formula for mild solutions of SPDEs, introduced in  Albeverio et al.  \cite{AGMRS2017}  and  written in terms of Yosida approximation

\begin{theorem}\label{Corollary Yosida app}[Corollary 3.7. \cite{AGMRS2017} ]
Assume conditions A)- E)  and  all the hypotheses in Theorem \ref{thm-Ito} are satisfied. Then 
\begin{equation} \lim_{n \to \infty} |\mathcal{L}\mathcal{H}(s,X^{n,\xi}_s)-  \mathcal{L}_n\mathcal{H}(s,X^{n,\xi}_s)| \quad P -a.s.  \end{equation} 
\end{theorem}

\section{Some stability properties for solutions of SPDEs on Hilbert spaces } In this Section we discuss how  the  It\^{o} Formula in Theorem \ref{thm-Ito} was applied by Albeverio et al. \cite{AGMRS2017} to  establish  through a Lyapunov function approach stability properties for the mild solution of  (\ref{SPDE-Hilbert}) converging to a unique invariant measure.  \\

\noindent {\bf Assumption} We assume in the whole Section that conditions A) -E) are satisfied.\\

 \noindent The mathematical tools introduced in  \cite{AGMRS2017} have been later extended in  \cite{FFRSch2020} to analyze the limiting behaviour of  mild solutions of SPDEs with multiple invariant measure. This will however not be  discussed here, due to a problem of space.

 \noindent  We start to recall some definition related to the Lyapunov function approach  presented in \cite{MW2011} (see also the PhD thesis of the second author L. Wang) as well as \cite{GMbook}, \cite{MRbook},  \cite{AGMRS2017}.

\begin{definition}\label{defExbddinmeansquare}
We say that the solution of (\ref{SPDE-Hilbert})  is exponentially stable in the mean square sense if there exists $c, \epsilon>0$ such that for all $t>0$ and  $\xi$$\in L^2(\Omega,\mathcal{F}_0,\mathbb{P};H)$   
\begin{equation}\label{ineq exp stable}
\mathbb{E}[\|X_t^\xi \|^2]\leq c e^{-\epsilon t} \mathbb{E}[ \|\xi\|^2]
\end{equation}
\end{definition}

\begin{definition}\label{DefLyapunov}
Let $\mathcal{L}$ be defined as in (\ref{Lyapunov}). A function ${\mathcal{H}}:\, H \to \mathbb{R}$ $\in C^2(H;\mathbb{R})$ is a Lyapunov function for the SPDE  (\ref{SPDE-Hilbert}) if it satisfies the following conditions:
\begin{enumerate}
\item[I.] There exist finite constants $c_1$, $c_2$$>0$ such that for all $ x \in H$
$$c_1 \|x\|^2 \leq {\mathcal{H}}(x) \leq c_2 \|x\|^2$$
\item[II.] There exists a constant $c_3>0$ such that 
$$\mathcal{L}\mathcal{H}(x) \leq - c_3 \mathcal{H}(x) \quad \forall  x \in \mathcal{D}(A)$$
\end{enumerate} 
\end{definition}

\noindent In Albeverio et al. \cite{AGMRS2017}  we proved the following Theorem

\begin{theorem} \label{Stable with Lyapunov}  [\cite {AGMRS2017} ] Assume that there exists a function $\mathcal{H}$ $\in C^2(H;\mathbb{R})$ which is a Lyapunov function for   the SPDE  (\ref{SPDE-Hilbert}) and  the hypotheses a) and b)  in Theorem \ref{thm-Ito} are  satisfied. Then the mild solution of    (\ref{SPDE-Hilbert})  is exponentially stable in the mean square sense. Moreover the constants in (\ref{ineq exp stable}) can be chosen so that   $c=\frac{c_2}{c_1}$ and $\epsilon=c_3$.
\end{theorem}

\noindent Remark that for the case  $\mathcal{H}$ $\in C_b^2(H;\mathbb{R})$ a proof can be found in [\cite{MW2011} Theorem 4.2] (see also [\cite{RZ2006} Section 7.1] and for the Gaussian case [\cite{GMbook} Theorem 6.4]. The results are stated there for the Yosida approximants.

\begin{proof}  
Since all the hypotheses of Theorem \ref{thm-Ito} are satisfied, It\^o formula can be applied to the Yosida approximation.
\begin{equation}
e^{c_3 t} \mathbb{E}[\mathcal{H}(X_t^{n,\xi}) - \mathcal{H}(R_n \xi)] =\mathbb{E}\left[ \int_0^t  e^{c_3 s} c_3 (\mathcal{H}(X_s^{n,\xi})+ \mathcal{L}_n \mathcal{H} (X_s^{n,\xi})) ds\right]
\end{equation}
From condition I it follows 
\begin{equation}
c_3 \mathcal{H}(X_s^{n,\xi})  + \mathcal{L}_n\mathcal{H}(X_s^{n,\xi}) \leq - \mathcal{L}\mathcal{H}(X_s^{n,\xi})+ \mathcal{L}_n\mathcal{H}(X_s^{n,\xi})  
\end{equation}
\begin{equation}
e^{c_3 t} \mathbb{E}[\mathcal{H}(X_t^{n,\xi}) - \mathcal{H}(R_n \xi)] \leq \mathbb{E}\left[ \int_0^t  e^{c_3 s}  (- \mathcal{L}\mathcal{H}(X_s^{n,\xi})+ \mathcal{L}_n \mathcal{H} (X_s^{n,\xi})) ds\right] 
\end{equation}
From Theorem \ref{imldtm3} and Theorem \ref{Corollary Yosida app} it follows $e^{c_3 t} \mathbb{E}[\mathcal{H}(X_t^{\xi})] \leq \mathbb{E}[ \mathcal{H}( \xi)]$. Condition II implies then 
\begin{equation}c_1 \mathbb{E}[\|X_t^{\xi}\|^2] \leq \mathbb{E}[\mathcal{H}(X_t^{\xi})] \leq e^{-c_3 t}\mathbb{E}[ \mathcal{H}( \xi)] \leq c_2  e^{-c_3 t} \mathbb{E}[\|\xi\|^2]
\end{equation}
and hence 
\begin{equation}\mathbb{E}[\|X_t^{\xi}\|^2] \leq \frac{c_2}{c_1} e^{-c_3 t} \mathbb{E}[\|\xi\|^2]
\end{equation}
The statement follows by choosing $c=\frac{c_2}{c_1}$ and $\epsilon=c_3$. 
\end{proof}

\noindent Using Theorem \ref{Stable with Lyapunov} we can provide an easy proof of   the following statement, known in the literature from e.g. [Section 16, \cite{DZbook}] and [Chapter 11, Section 5, \cite{PZbook}].

\begin{theorem}\label{Thdissipativity}
Assume that the conditions A) - E)  are  satisfied  for  (\ref{SPDE-Hilbert}),  and the following conditions hold
\begin{enumerate}
\item[i)] $A$ satisfies the ''dissipativity condition'' , i.e there exists $\alpha$ $>0$ such that 
\begin{eqnarray}
&<Ax-Ay,x-y>+ <F(x) -F(y), x-y>   \nonumber \\&   \leq -\alpha \|x-y\|^2 \quad \forall x,y\in \mathcal{D}(A);
\end{eqnarray}

\item[ii)] $\epsilon:= 2 \alpha -L_f$ $>0$.

\item[iii)] $\forall z \in H$
 $ \int_A \| f(u,z) \|^4 \beta(du)< \infty$
\end{enumerate}

\noindent Then  for all $\xi,\, \eta\, \in L^2(\Omega,\mathcal{F}_0,\mathbb{P};H)$
\begin{equation}\label{stability Difference}
\mathbb{E}[\|X_t^{\xi}-X_t^\eta\|^2] \leq e^{-\epsilon  t} \mathbb{E}[\|\xi -\eta\|^2]  \quad \forall  t>0 
\end{equation}
\end{theorem}

\begin{proof} The stochastic  process $X_\cdot^\xi - X_\cdot ^\eta$ is the mild solution of 

\begin{eqnarray} \label{SPDE-Difference}
d(X_t^\xi-X_t^\eta)  &=&   A (X_t^\xi-X_t^\eta) dt+ (F(X_t^\xi) - F(X_t^\eta)) dt \nonumber\\& + &\int_{H} (f(u,X_t^\xi) - f(u,X_t^\eta)) q(dt,du) \\ X^\xi_0 - X^\eta_0&=&  \xi -\eta .
\end{eqnarray}

\noindent Condition iii) implies that all hypothesis of Theorem \ref{thm-Ito} are satisfied for $\mathcal{H}(x,y):=\|x-y\|^2$. Moreover, according to   the definition  of $\mathcal{L}$ in (\ref{Lyapunov}), we have 
\begin{eqnarray*} &\mathcal{L}\|x-y\|^2:=2<x-y, A(x-y)> +2 <x-y, F(x) - F(y)> \nonumber \\&  +\int_H \| f(u,x) - f(u,y)\|^2  \beta(du) 
 \end{eqnarray*}
where we used that

\begin{eqnarray*}
&\|x-y + f(u,x)- f(u,y)\|^2-\|x-y\|^2 \\&- 2<x-y,  f(u,x)- f(u,y)>=\|  f(u,x) -  f(u,y)\|^2 
\end{eqnarray*}

\noindent Conditions i) and ii) imply that the function $\mathcal{H}(x,y):=\|x-y\|^2$ is a Lyapunov function for (\ref{SPDE-Difference})  with $c_1=c_2=1$ and $c_3=\epsilon$. Hence  $X_\cdot^{\xi} - X_\cdot ^{\eta}$   is exponentially stable in the mean square sense.
\end{proof}

\smallskip 
\noindent We denote by $p_t^*$ the adjoint operator to $p_t$  defined  in (\ref{transition prob}), i.e. 
\[
 p_t^* \rho(dx) = \int_H p_t(y,dx) \rho(dy), \quad t \geq 0.
\]
Recall that a probability measure $\pi$ on $(H, \mathcal{B}(H))$ is called \textit{invariant measure} for the semigroup $(p_t)_{t \geq 0}$ if and only if $p_t^* \pi = \pi$ holds for each $t \geq 0$.
Let $\mathcal{P}_2(H)$ be the space of Borel probability measures $\rho$ on $(H, \mathcal{B}(H))$ with finite second moments. Recall that $\mathcal{P}_2(H)$ is separable and complete when equipped with the \textit{Wasserstein-2-distance}
\begin{equation}\label{eq: Wasserstein distance}
 \mathrm{W}_2(\rho, \widetilde{\rho}) = \inf_{ G \in \mathcal{H}(\rho, \widetilde{\rho})} \left( \int_{H \times H} \| x - y \|_H^2 G(dx,dy) \right)^{\frac{1}{2}}, \quad \rho, \widetilde{\rho} \in \mathcal{P}_2(H).
\end{equation}
Here $\mathcal{H}(\rho,\widetilde{\rho})$ denotes the set of all couplings of $(\rho, \widetilde{\rho})$, i.e. Borel probability measures on $H \times H$ whose marginals are given by $\rho$ and $\widetilde{\rho}$, respectively, see \cite[Section 6]{Villanibook} for a general introduction to couplings and Wasserstein distances.\\

\noindent 
As a consequence of our key stability estimate (\ref{stability Difference}) we can provide, by following  the proof of Theorem 4.1 in \cite{FFRSch2020},  a proof for the existence and uniqueness of a unique limiting distribution in the spirit of classical results such as
\cite[Section 16]{PZbook}, \cite[Chapter 11, Section 5]{DZbook}, and \cite{Rusinek10}. 
\begin{theorem}
Assume that the conditions A) - E)  are  satisfied  for  (\ref{SPDE-Hilbert}),  and the conditions i)-iii) in Theorem \ref{Thdissipativity} hold. Then 

 \begin{equation}\label{eq: 00}
  \mathrm{W}_2(p_t^* \rho, p_t^* \widetilde{\rho}) \leq  \mathrm{W}_2(\rho, \widetilde{\rho})\mathrm{e}^{- \epsilon t/2}, \quad t \geq 0,
 \end{equation}
 holds for any $\rho, \widetilde{\rho} \in \mathcal{P}_2(H)$.
 In particular, the Markov process determined by  (\ref{SPDE-Hilbert}) has a unique invariant measure $\pi$. This measure has finite second moments and it holds that
 \begin{equation}\label{eq: 01}
  \mathrm{W}_2(p_t^*\rho, \pi) \leq \mathrm{W}_2(\rho, \pi) \mathrm{e}^{- \epsilon t/2 }, \quad t \geq 0,
 \end{equation}
 for each $\rho \in \mathcal{P}_2(H)$.
\end{theorem}
\begin{proof}
 From Theorem \ref{Thdissipativity} it follows
 \begin{equation}
     \E[ \| X_t^x - X_t^y\|_H^2 ] \leq \mathrm{e}^{-\epsilon t} \| x-y\|_H^2, \quad x,y \in H. \nonumber
 \end{equation} 
 Using the definition of the Wasserstein distance, we conclude that 
 \[
  \mathrm{W}_2(p_t^* \delta_x, p_t^* \delta_y) \leq \left( \E[ \| X_t^x - X_t^y\|_H^2 ] \right)^{1/2}
  \leq \|x-y\|_H \mathrm{e}^{- \epsilon t/2}.
 \]
 The latter one readily yields (\ref{eq: 00}). 
 Finally, the existence and uniqueness of an invariant measure as well as (\ref{eq: 01}) can be derived from (\ref{eq: 00}) combined with a standard Cauchy argument. 
\end{proof} 
In  \cite{FFRSch2020} we introduced  a ``generalized dissipativity condition'' and  studied SPDEs with  multiple invariant measures. There we developed  further the methods presented in this Section, which have been mainly derived from Albeverio et al. \cite{AGMRS2017} in combination with the results obtained in \cite{AMR2009},  \cite{MRbook}.

  \bigskip
\paragraph{\textbf{Acknowledgment.}}
I thank Peter Kuchling and Baris Ugurcan for a careful reading of part of this article. 

  \bigskip
\paragraph{\textbf{Comment  by Barbara R\"udiger}} My co-author and friend V. Mandrekar (Atma) passed away the 23  June 2021. A couple of days before his departure he contacted me through email to make sure the  procedure for the submission of this article would be successful. The invitation to contribute  to this Volume, dedicated to Sergio Albeverio, was accepted by him with enthusiasm.\\
\noindent Atma and Sergio had, to my feeling, a deep respect for each other and, despite  the geographic distance, a solid friendship. I think that this friendship and respect is also due to common aspects they have in their character and soul: both are very generous in sharing with other scientists their original ideas. Both trust in youngsters and enjoy knowing  that they can contribute to these  with their own developments and ideas, as well. This way they both are friends,  supporters, coaches and co -authors to many young (and in the meanwhile older) mathematicians and physists. I feel very lucky to be among them.

\bibliographystyle{amsplain}
\phantomsection\addcontentsline{toc}{section}{\refname}\bibliography{Bibliography}

\end{document}